\documentclass[a4paper]{article}
\usepackage{amssymb,enumerate,amsthm}
\usepackage{color}
\usepackage{graphicx}
\usepackage{xspace}
\usepackage{amsmath}
\usepackage{epsfig}
\usepackage{tikz}
\usepackage{aliascnt}
\usepackage[pdfdisplaydoctitle,menucolor=orange!40!black,filecolor=magenta!40!black,urlcolor=blue!40!black,linkcolor=red!40!black,citecolor=green!40!black,colorlinks]{hyperref}

{\itshape}{\rmfamily}

\newtheorem{definition}{Definition}

\newtheorem{theorem}[definition]{Theorem}

\newtheorem{lemma}[definition]{Lemma}

\newcommand{\boxicity}{{\sc{Boxicity}}\xspace}

\newcommand{\sm}{\setminus}

\newlength{\atextwidth}
\DeclareMathOperator*{\cart}{\times}

\DeclareMathOperator*{\boxOp}{box}
\DeclareMathOperator*{\bwOp}{bw}
\DeclareMathOperator*{\pwOp}{pw}
\DeclareMathOperator*{\twOp}{tw}

\newcommand{\problemdec}[3]{
  \vspace{1mm}
\noindent\fbox{
  \begin{minipage}{\atextwidth}
  \begin{tabular*}{\textwidth}{@{\extracolsep{\fill}}lr} #1 \\ \end{tabular*}
  {\bf{Input:}} #2  \\
  {\bf{Question:}} #3
  \end{minipage}
  }
  \vspace{1mm}
}

\usepackage{todonotes}

\usetikzlibrary{chains,positioning,backgrounds,shapes,fit}
\tikzstyle{background}=[rectangle,fill=red!10,inner sep=0.2cm,rounded corners=1mm,minimum size=1cm]

\tikzstyle{vertex}=[circle,draw,minimum size=20pt,inner sep=0pt]
\tikzstyle{edge} = [draw,thick,-]

\newcommand{\Gsegment}[1]{\ensuremath{G_{#1}}}

\begin{document}


\title{Structural parameterizations for boxicity}
\author{Henning Bruhn, Morgan Chopin, Felix Joos and Oliver Schaudt}
\date{}


%
\maketitle

\begin{abstract}
The boxicity of a graph $G$  is the least integer $d$ such that $G$ has an
intersection model of axis-aligned $d$-dimensional boxes.
\boxicity, the problem of deciding whether a given graph $G$ has boxicity at most~$d$, 
is NP-complete for every fixed $d \ge 2$.
We show that \boxicity is fixed-parameter tractable when parameterized by the cluster vertex deletion number of the input graph.
This generalizes the result of Adiga~et~al.,  that \boxicity is fixed-parameter tractable in the vertex cover number.

Moreover, we show that \boxicity admits an 
additive $1$-approximation when parameterized by the pathwidth of the input graph.

Finally, we provide evidence in favor of a conjecture of Adiga et al.~that \boxicity remains NP-complete when parameterized by the treewidth.
\end{abstract}


\section{Introduction}

Every graph $G$ can be represented as an intersection graph of axis-aligned boxes in $\mathbb R^d$, 
provided $d$ is large enough. The \emph{boxicity} of $G$,
introduced by~Roberts~\cite{roberts69}, 
 is the smallest dimension $d$ for which 
this is possible. 
We denote 
the corresponding decision problem by \boxicity:
given~$G$ and~$d\in\mathbb N$, determine whether~$G$ has boxicity at most~$d$.

Boxicity has received a fair 
amount of attention.
This is partially due to the wider context of graph representations,
but also because graphs of low boxicity are interesting from an algorithmic point 
of view. 
While 
many hard problems remain so for graphs of bounded boxicity, some become  
solvable in polynomial time, notably max-weighted clique (as observed by Spinrad~\cite[p.~36]{Spi03}).

Cozzens~\cite{cozzens81} showed that
\boxicity is NP-complete.
To cope with this hardness result, several 
authors~\cite{adiga12, adiga10-b, ganian11}
studied the parameterized complexity of \boxicity. Since the problem remains
NP-complete for constant $d\geq 2$ (Yannakakis~\cite{yannakakis82} 
and~Kratochv{\'i}l~\cite{kratochvil94}),  boxicity itself is ruled out as 
parameter. Instead more structural parameters 
have been considered. Our work follows this line. We prove:

\begin{theorem}\label{dtcthm}
\boxicity is fixed-parameter tractable when parameterized
by cluster vertex deletion number.
\end{theorem}
The
\emph{cluster vertex deletion} number
is
the minimum number of vertices that have to be deleted to 
 get a disjoint union of complete graphs or \emph{cluster graph}.
As discussed by Doucha and Kratochv\'il~\cite{doucha12}
cluster vertex deletion is an
intermediate parameterization between vertex cover and cliquewidth.
A \emph{$d$-box representation} of a graph $G$ is a representation of $G$ 
as intersection graph of axis-aligned boxes in $\mathbb R^d$.

\begin{theorem}\label{pwalgothm}
\sloppy Finding a $d$-box representation of~$G$ such that~${d \leq \boxOp(G) + 1}$ can be done in~$f(\pwOp(G)) \cdot |V(G)|$ time where~$\pwOp(G)$ is the pathwidth of~$G$.
\end{theorem}

A natural parameter for \boxicity is the treewidth $\twOp(G)$ of a graph $G$, in particular as 
Chandran~and~Sivadasan~\cite{chandran07} proved that
$\boxOp(G) \leq \twOp(G) + 2$. However, 
Adiga, Chitnis and Saurabh~\cite{adiga10-b} conjecture that \boxicity is NP-complete on graphs of bounded treewidth.
Our last result provides evidence in favor of this conjecture.
For this, we mention the observation of Roberts~\cite{roberts69}
that a graph $G$ has boxicity~$d$ if and only if $G$ can be expressed as the intersection
of $d$ interval graphs.

\begin{theorem}\label{bandthm}
There is an infinite family of graphs $G$ of boxicity~$2$ and bandwidth $\mathcal O(1)$ such that, among any pair of interval graphs whose intersection is $G$, at least one has treewidth $\Omega(|V(G)|)$.
\end{theorem}

Why do we see the result as evidence? 
An algorithm solving \boxicity on graphs of bounded treewidth 
(or even stronger, of bounded bandwidth) is likely to exploit the local structure 
of the graph in order to make dynamic programming work. Yet, Theorem~\ref{bandthm}
implies that this locality may be lost in some dimensions, which constitutes a serious obstacle
for any dynamic programming based approach. We discuss this in more detail in Section~\ref{sec:bandthm}. 

\medskip
Figure~\ref{fig:paramsmap} summarizes previously known
parameterized complexity results on boxicity along with those obtained in this article. 
Adiga et al.~\cite{adiga10-b} initiated this line of research when 
they parameterized \boxicity by the minimal size $k$ of a \emph{vertex cover}
in order to give an~$2^{O(2^k k^2)} \cdot n$-time algorithm, where $n$ 
denotes the number of vertices of the input graph, as usual. 
Adiga et al.\ also described an approximation algorithm that, in time $2^{O(k^2 \log k)}\cdot n$,
returns a box representation of at most $\boxOp(G) + 1$ dimensions. 
Both results were extended by Ganian~\cite{ganian11} to the less restrictive parameter \emph{twin cover}.
Our Theorem~\ref{dtcthm} includes Ganian's.

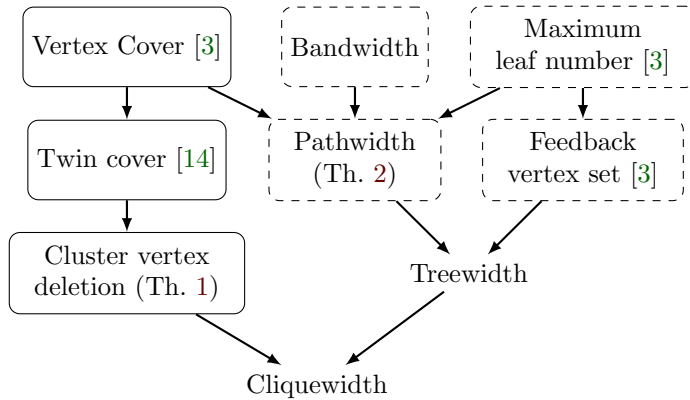
\begin{figure}[t]
\centering
\newcommand{\tworows}[2]{\begin{tabular}{c}{#1}\\{#2}\end{tabular}}
\newcommand{\threerows}[3]{\begin{tabular}{c}{#1}\\{#2}\\{#3}\end{tabular}}
\begin{tikzpicture}[->,>=latex]
\tikzstyle{pstyle} = []
\tikzstyle{fpt} = [draw, rectangle, rounded corners, minimum height=3em]
\tikzstyle{approx} = [draw, dashed, rectangle, rounded corners, minimum height=3em]
\tikzstyle{hypohard} = []

\node[fpt] (vc) at (2.5,4) {Vertex Cover~\cite{adiga10-b}};
\node[approx] (bw) at (5.5,4) {Bandwidth};
\node[approx] (mln) at (8.5,4) {\tworows{Maximum}{leaf number~\cite{adiga10-b}}};

\node[fpt] (tc) at (2.5,2.5) {Twin cover~\cite{ganian11}};
\node[approx] (pw) at (5.5,2.5) {\tworows{Pathwidth}{(Th.~\ref{pwalgothm})}};
\node[approx] (fvs) at (8.5,2.5) {\tworows{Feedback}{vertex set~\cite{adiga10-b}}};

\node[fpt] (cvd) at (2.5,1) {\tworows{Cluster vertex}{deletion~(Th.~\ref{dtcthm})} };
\node[hypohard] (tw) at (7,1) {Treewidth};

\node[hypohard] (cw) at (5,-0.5) {Cliquewidth};

\draw[thick] (vc) edge (tc) edge (pw);
\draw[thick] (bw) edge (pw);
\draw[thick,<-] (cvd) edge (tc);

\draw[thick] (mln) edge (fvs) edge (pw);
\draw[thick,<-] (tw) edge (fvs) edge (pw);

\draw[thick,<-] (cw) edge (cvd) edge (tw);
\end{tikzpicture}
\caption{Navigation map through our parameterized complexity results for \boxicity. An arc from a parameter~$k_2$ to a parameter~$k_1$ means that there exists some function~$h$ such that~$k_1 \leq h(k_2)$. A rectangle means fixed-parameter tractability for this parameter and a dashed rectangle means an approximation algorithm with running time~$f(k)\cdot n^{O(1)}$ is known.}
\label{fig:paramsmap}
\end{figure}

Other structural parameters that were considered by Adiga~et~al.\ for parameterized approximation algorithms
are the size of a \emph{feedback vertex set} -- the minimum number of vertices 
that need to be deleted to obtain a forest -- and \emph{maximum leaf number} -- the maximum number of leaves in a spanning tree of 
the graph. They proved that 
finding a $d$-box representation of a graph~$G$ such that~$d \leq 2\boxOp(G) + 2$ (resp.~$d \leq \boxOp(G) + 2$) can be done in~$f(k)\cdot |V(G)|^{O(1)}$ time (resp.~$2^{O(k^3 \log k)}\cdot |V(G)|$ time) where~$k$ is the size of a feedback vertex set (resp. maximum leaf number).
In~\cite{adiga12}, Adiga, Babu, and Chandran generalized these approximation algorithms 
to parameters of the type
``distance to~$\mathcal{C}$'', where $\mathcal C$ is any graph class of bounded boxicity.
More precisely, the parameter measures the minimum number of vertices
 whose deletion results in a graph that belongs~$\mathcal{C}$.

The algorithm of Theorem~\ref{pwalgothm} generalizes the approximation algorithm 
for the parameter vertex cover number, 
and improves the guarantee bound of the approximation algorithm for the 
parameter maximum leaf number.

\medskip

There is merit in studying approximation algorithms from a parameterized perspective:
not only is \boxicity NP-complete, but the associated minimization problem
cannot be approximated in polynomial time  
within a factor of $n^{\frac{1}{2} - \varepsilon}$ for any~$\varepsilon > 0$ 
even when the input is restricted to bipartite, co-bipartite or split graphs (provided~NP$\neq$ZPP).
This is a result due to Adiga, Bhowmick and Chandran~\cite{adiga10}. 
There is, however, an approximation algorithm with factor $o(n)$ for general graphs; 
see Adiga~et~al.~\cite{adiga12}.

\medskip

While Roberts~\cite{roberts69} was the first to study the boxicity parameter, 
he was hardly the first to consider box representations of graphs. 
Already in 1948 Bielecki~\cite{Bie48} asked, here phrased in modern terminology, 
whether triangle-free graphs of 
boxicity $\leq 2$ had bounded chromatic number. This was answered affirmatively
by Asplund and Gr\"unbaum in~\cite{AG60}. Kostochka~\cite{Kos04}
treats this question in a much more general setting.

Following Roberts who proved that~$\boxOp(G) \leq \frac{n}{2}$,
other authors obtained bounds for  boxicity. 
Esperet~\cite{esperet09}, for instance, showed that~$\boxOp(G) \leq \Delta(G)^2 + 2$,
while Scheinerman~\cite{scheinerman84}  established that 
every outerplanar graph has boxicity at most two. 
This, in turn, was extended by Thomassen~\cite{thomassen86}, who showed that 
planar graphs have boxicity at most three.

\medskip

In the next section, we will give formal definitions of the necessary concepts
for this article. We prove our main results in Sections~\ref{sec:dtcthm}--\ref{sec:bandthm}.
Finally, we discuss the impact and limitations of our results in Section~\ref{sec:discussion},
where we also outline some future directions for research.

\section{Preliminaries}\label{sec:prelim}

\paragraph{Graph terminology.} We follow the notation of Diestel~\cite{diestelBook10},
where also all basic definitions concerning graphs may be found. 

Let $X$ be some finite set. With a slight abuse of notation, we consider a collection $I=([\ell_v,r_v])_{v\in X}$ of closed intervals in the real line to be an \emph{interval graph}: $I$ has vertex set $X$, and two of its vertices
$u$ and $v$ are adjacent if and only if the corresponding intervals $[\ell_u,r_u]$
and $[\ell_v,r_v]$ intersect. By perturbing the endpoints of the intervals we can 
ensure that no two intervals have a common endpoint, and that for every interval 
the left endpoint is distinct from the right endpoint. We always tacitly 
assume the intervals to be of that form.
Fig.~\ref{fig:forbidden} shows the family of forbidden subgraphs for the class of interval graphs.

The \emph{bandwidth} of a graph $G$, say with vertex set $V(G) = \{v_1,v_2,\ldots,v_n\}$, is the least number $k$ for which the vertices of $G$ can be labeled with distinct integers $\ell(v_i)$ such that $k = \max\{|\ell(v_i)-\ell(v_j)| : v_iv_j\in E\}$. 
Equivalently, it is the least integer $k$ for which the vertices of $G$ can be placed at distinct integer points on the real line such that the length of the longest edge is at most $k$.
We denote the bandwidth of a graph $G$ by $\bwOp(G)$.

We define a \emph{path decomposition} of a graph $G$ as a set~$\mathcal{W} = \{W_1, \ldots, W_t\}$ of subsets of $V(G)$ called \emph{bags} such that the following conditions are met.
\begin{enumerate}
\item $\bigcup_{i = 1}^t W_i = V(G)$.
\item For each $uv \in E(G)$, there is an~$i \in \{1,\ldots,t\}$ such that $u, v \in W_i$.
\item For each~${v \in V(G)}$, if~$v \in W_i \cap W_j$ for some~$i,j \in \{1,\ldots,t\}$, then~$v \in W_k$ with~$i \leq k \leq j$.
\end{enumerate}
The \emph{width} of a path decomposition is~$\max_{i} |W_i| - 1$.
The \emph{pathwidth}~$\pwOp(G)$ of a graph~$G$ is the minimum width over all possible path decompositions of~$G$.
Equivalently, $\pwOp(G)$ is the minimum size of the largest clique of any interval supergraph of $G$, minus 1.

     \begin{figure}[ht]
      \centering
      \includegraphics[scale=0.7]{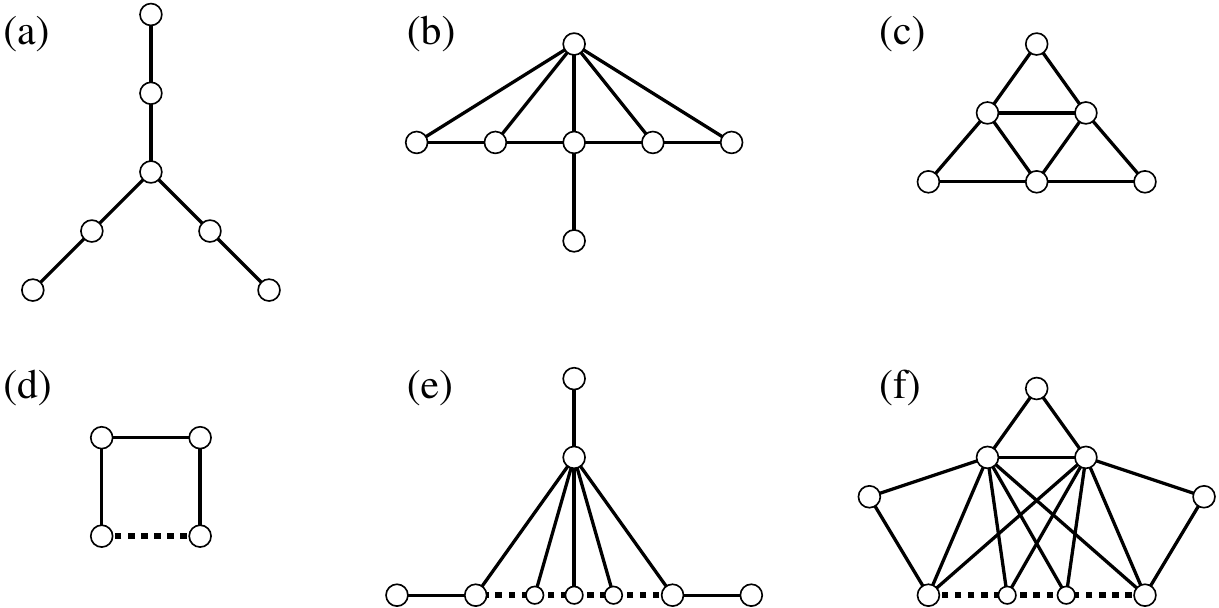}
      \caption{Forbidden induced subgraphs for interval graphs; the dashed
paths may have any length $\geq 1$}\label{fig:forbidden}
      \end{figure}

The \emph{treewidth} of a graph $G$, denoted $\twOp(G)$, is the minimum size of the largest clique of any chordal supergraph of $G$, minus 1.
For the purpose of our paper it is important to remark that for every graph~$G$
we have~$\twOp(G) \leq \pwOp(G) \leq \bwOp(G)$.


\paragraph{Parameterized complexity.} 
A decision problem parameterized by a problem-specific parameter~$k$ is called \emph{fixed-parameter tractable} 
if there exists an algorithm that solves it in time~$f(k) \cdot n^{O(1)}$, where~$n$ is the instance size. 
The function~$f$ is typically super-polynomial and only depends on~$k$. 
One of the main tools to design such algorithms is the
\emph{kernelization} technique. 
A kernelization algorithm transforms in
polynomial time an instance~$I$ of a given problem parameterized
by~$k$ into an equivalent instance~$I'$ of the same problem
parameterized by~$k' \leq k$ such that the size of~$I'$ is bounded by~$g(k)$ for some
computable function~$g$. The instance~$I'$ is called a
\textit{kernel} of size~$g(k)$. The following folklore result is well known. 
\begin{theorem}
A parameterized problem~$P$ is fixed-parameter tractable if and only if~$P$ has a kernel.
\end{theorem}
In the remainder of this paper, the kernel size is expressed in terms of the number of vertices.

For more background on parameterized complexity the reader is referred to
Downey and Fellows~\cite{DowneyF13}. 

\paragraph{Problem definition.} We call an \emph{axis-aligned $d$-dimensional box} (or \emph{$d$-box}) a cartesian product of~$d$ closed real intervals. A \emph{$d$-box representation} of a graph~$G$ is a mapping that maps every vertex~$v \in V(G)$ to a $d$-box~$B_v$ such that two vertices~$u,v \in V(G)$ are adjacent if and only if their associated boxes have a non-empty intersection. The \emph{boxicity} of~$G$, denoted by $\boxOp(G)$, is the minimum integer~$d$ such that~$G$ admits a $d$-box representation. We consider the following problem.

\problemdec{\boxicity}{A graph~$G$ and an integer~$d$.}{Is $\boxOp(G) \leq d$?}

Given a~$d$-box representation of~$G$, we denote by~$[\ell_i(v),r_i(v)]$ the interval representing~$v$ in the $i$-th dimension.

Throughout the article, 
we make frequent use of the reformulation of boxicity in terms of interval graphs: 

\begin{theorem}[Roberts~\cite{roberts69}]
The boxicity of a graph $G$ is equal to
the smallest integer $d$ so that $G$ can be  expressed 
as the intersection of $d$ interval graphs.
\end{theorem}

\section{Proof of Theorem~\ref{dtcthm}}\label{sec:dtcthm}

Theorem~\ref{dtcthm} follows immediately from the following lemma:

\begin{lemma}\label{dtckernel}
\boxicity admits a kernel of at most~$k^{2^{O(k)}}$ vertices, 
where~$k$ is the cluster vertex deletion number of the input graph.
\end{lemma}

In the course of this section, we present a sequence of lemmas in order to prepare the proof of our main lemma above.

Two adjacent vertices $u,v$ in a graph $G$ are \emph{true twins} if $u$ 
and $v$ have the same neighbourhoods in $G-\{u,v\}$. 
As observed by Ganian~\cite{ganian11}, deleting one of two true twins does not 
change the boxicity.

\begin{lemma}\label{notwinslem}
Let $u,v$ be true twins of a graph $G$. Then $\boxOp(G)=\boxOp(G-u)$.
\end{lemma}

We remark, without proof, that there is also a reduction for
false twins (those that are non-adjacent): if there are at least three of them, 
then one may be deleted without changing the boxicity. We will not, however, make use of this 
observation.

Recall that a \emph{cluster graph} is the disjoint union of complete graphs, called \emph{clusters}.
In what follows, we implicitly identify a cluster with its vertex set.

Let $G-X$ be a cluster graph for some $X\subseteq V(G)$. 
We call two clusters $C,C'$ of $G-X$ \emph{equivalent}
if there is a bijection $C\to C'$, $v\mapsto v'$, such that 
$N_G(v)\cap X=N_G(v')\cap X$. 
Observe that, if $G-X$ has no true twins, then two clusters $C$ and $C'$ are equivalent
if and only if $\{N_G(u)\cap X:u\in C\}=\{N_G(v)\cap X:v\in C'\}$.
 
\begin{lemma}\label{smallclusters}
Let $G$ be a graph without true twins, and let $X$ be a set of $k$ vertices so that $G-X$
is a cluster graph. Then every cluster in $G-X$ contains at most~$2^k$ vertices.
\end{lemma}
\begin{proof}
Consider a cluster $C$ of $G$. 
Then the number of sets $N_G(v)\cap X$, $v\in C$, is bounded 
by $2^k$. As $G$ has no true twins, no two vertices in $C$ may
have the same neighbourhood in $X$, which implies that $|C|\leq 2^k$.
\end{proof}

We also need the following result.

\begin{theorem}[Chandra and Sivadasan~\cite{chandran07}]\label{boxtwthm}
It holds that $\boxOp(G) \leq \twOp(G)+2$ for any graph~$G$.
\end{theorem}

In particular, $\boxOp(G) \leq \pwOp(G)+2$ for any graph~$G$. 

\begin{lemma}\label{clusterdellem}
Let $G$ be a graph without true twins, and let $X$ be a set of $k$ vertices so that $G-X$ is a cluster graph.
Moreover, let $\mathcal D$ be an equivalence class of clusters with
$|\mathcal D|\geq 2(2k+2)^{2^{k+1}(2^k+k+1)}$.
For every $C^*\in \mathcal D$, $\boxOp(G)=\boxOp(G-C^*)$.
\end{lemma}
\begin{proof}
As deleting vertices may only decrease the boxicity, it suffices to prove that 
$\boxOp(G)\leq\boxOp(G-C^*)$.

Set $H=G-C^*$, $d=\boxOp(H)$, $k=|X|$ and $\mathcal C= \mathcal D \setminus \{C^*\}$.
We claim that
\begin{equation} \label{cl:box-bounded}
d=\boxOp(H) \leq 2^k + k + 1.
\end{equation}
Indeed, define a path decomposition with a bag~$W_C$ for every cluster $C$ of~$H-X$ such that~$W_C = X\cup C$. This gives a path decomposition of $H$ with width at most $k+2^k-1$, by~Lemma~\ref{smallclusters}. 
Theorem~\ref{boxtwthm} now implies~\eqref{cl:box-bounded}.

For the sake of simplicity, let us introduce the following notions.
Fix a $d$-box representation of $H$. The set of \emph{corners} of a box of a vertex is the cartesian product~$\cart_{i=1}^d \{\ell_i(v),r_i(v)\}$.
By rescaling every dimension (compare Lemma~\ref{gridifylem}), we can ensure
that every endpoint of an interval of a vertex in $X$ lies in $\{1,2,\ldots,2k\}$.
Thus every corner  of a box of $X$ lies in the grid $\{1,2,\ldots,2k\}^d$.
We may moreover assume that every other box of $H$ is contained in $[0,2k+1]^d$.
Points of $\{0,1,\ldots,2k+1\}^d$ we call \emph{grid points}, 
and any set $[z_1,z_1+1]\times\ldots\times[z_d,z_d+1]$, where $z_i\in\{0,\ldots ,2k\}$, 
is a \emph{grid cell}.
In each dimension~$i$ we say that the grid induces the \textit{grid intervals} $[0,1]$, $[1,2]$, \ldots, $[2k,2k+1]$.
A box of a vertex in $H-X$ is a \emph{cluster box}. 

By perturbing the boxes slightly we may always assume that
\begin{equation}\label{perturb}
\begin{minipage}[c]{0.8\textwidth}\em
if $s$ is a corner of a cluster box of a cluster $C$ of $H-X$, 
and if $t$ is a corner  of the box of any vertex $z\in V(H-C)$
then $s_i\neq t_i$ for all dimensions $i=1,\ldots, d$.
\end{minipage}\ignorespacesafterend 
\end{equation} 
Moreover, we may assume that any corner of a cluster box lies in the interior of a grid cell.
A cluster box that does not contain any grid point we call a \emph{thin box}.

      \begin{figure}[ht]
      \centering
      \input{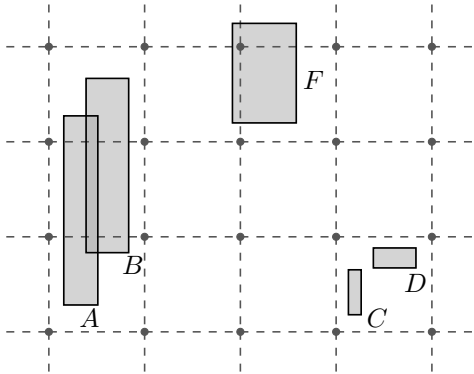}
      \caption{Boxes $A,B$ are in the same position, as are $C$ and $D$; $F$ is not thin.}\label{fig:thin}
      \end{figure}

We concentrate on \emph{thin} clusters, that is, clusters 
that consist of thin boxes only. 
We claim that
\begin{equation} \label{cl:bound-thin-boxes}
\text{\em 
at least $(2k+2)^{2^{k+1}(2^k+k+1)}$ clusters in
$\mathcal C$ are thin.}
\end{equation}
To prove this claim, observe that no grid point lies in a cluster box of two different clusters
as then two vertices in distinct clusters would be adjacent. Thus, there is at most 
one cluster per grid point so that one of its cluster boxes contains the grid point.
As,  by~\eqref{cl:box-bounded}, there are $(2k+2)^d\leq (2k+2)^{2^k+2k+1}$ 
grid points, it follows that $\mathcal C$ has 
at least $|\mathcal C|-(2k+2)^{2^k+2k+1}\geq (2k+2)^{2^{k+1}(2^k+k+1)}$ thin clusters.

\medskip
We say that two cluster boxes~$B$ and~$B'$ are in the \emph{same position} if every grid cell containing a corner of~$B$ also contains a corner of~$B'$ and vice versa (see~Figure~\ref{fig:thin}).
Note that if two vertices~$v, v' \in V(H)-X$ have boxes in the same position then $N_H(v) \cap X = N_H(v') \cap X$. (Here we use the fact that cluster boxes have their corner strictly in the interior of grid
cells.)

For every cluster $C\in\mathcal C$ we fix a point $p(C)$ that lies in every cluster box of~$C$:
such a point exists by the Helly property for boxes in $\mathbb R^d$.
We claim that, using this Helly point, 
we can modify our box representation of $H$ so that
\begin{equation}\label{thinify}
\begin{minipage}[c]{0.8\textwidth}\em
for all thin clusters $C\in\mathcal C$  and for each dimension
$i\in\{1,\ldots, d\}$ holds the following: if $p(C)$ and a corner  $t$ of a box of $C$
  lie in the same grid interval in dimension $i$, 
that is, if there is a $j$ so that $p_i(C),t_i\in [j,j+1]$, then $t_i=p_i(C)$.
\end{minipage}\ignorespacesafterend 
\end{equation} 

To achieve~\eqref{thinify}, we proceed as follows. 
Let $v$ be a vertex of any thin cluster~$C\in\mathcal C$.
Consider a dimension $i$ where $\ell_i(v)$ or $r_i(v)$ lie in the same grid interval 
as~$p_i(C)$.
Note that $\ell_i(v)\leq p_i(C) \leq r_i(v)$.
In dimension $i$, we shrink the box of $v$ in the following way:
if $\ell_i(v)$ lies in the same grid interval as~$p_i(C)$,
we replace  $\ell_i(v)$ by $p_i(C)$.
Similarly, if $r_i(v)$ lies in the same grid interval as~$p_i(C)$, 
we replace $\ell_i(v)$ by $p_i(C)$.
This procedure is illustrated in Fig.~\ref{fig:thinner}.

Since by shrinking a box we may only lose edges of the corresponding graph,
it suffices to show that every edge is still present.
Since the new box of $v$ still contains $p(C)$,
the vertex $v$ is still adjacent to every other vertex in $C$.
As we change the box of $v$ only within a grid interval, 
the old and the new box of $v$ are in the same position.
Thus,
we do not lose any edge from $v$ to $X$.
Performing this transformation iteratively for every box of $C$  in every dimension,
and for every thin cluster $C\in\mathcal C$, we
obtain a box representation of $H$ satisfying~\eqref{thinify}.

      \begin{figure}[ht]
      \centering
      \input{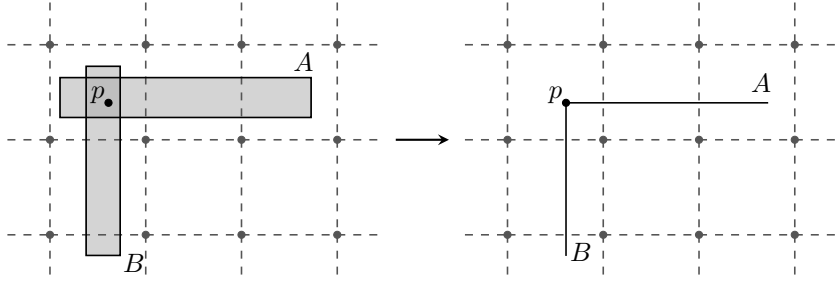}
      \caption{Shrinking the boxes}\label{fig:thinner}
      \end{figure}

\medskip
Next, we claim that
\begin{equation}\label{cl:cluster-topo-id}
\begin{minipage}[c]{0.8\textwidth}\em
there is a pair of distinct thin clusters $C,C' \in \mathcal C$
such that for every $v \in C$ and $v' \in C'$ with $N_H(v) \cap X = N_H(v') \cap X$, the boxes of $v$ and $v'$ are in the same position.
\end{minipage}\ignorespacesafterend 
\end{equation} 
Note that, as $C$ and $C'$ are equivalent, 
there is indeed a bijection between the vertices of $C$ and $C'$ 
that maps a vertex $v$ to $v' \in C'$ with $N_H(v) \cap X = N_H(v') \cap X$.

\medskip

Observe that for the endpoints~$\ell_i(v), r_i(v)$ of the interval representing a vertex~$v \in V(H)$ in the $i$-th dimension, there are at most
$(2k+1)^2$ many choices to select the grid intervals they lie in.
Thus, any set of thin boxes, pairwise not in the same position, has size at most $(2k+1)^{2d}$. 
Because $G$ is devoid of true twins, no cluster has two vertices whose boxes are in the 
same position. 

Recall that every cluster has at most $2^k$ vertices.
Thus, among any choice of more than
$(2k+1)^{2d\cdot 2^k}$ thin clusters there are two thin clusters satisfying~\eqref{cl:cluster-topo-id}. 
As $(2k+1)^{2d\cdot 2^k}\leq (2k+1)^{2(2^k+k+1))\cdot 2^k}$, by~\eqref{cl:box-bounded}, 
and since $\mathcal C$
contains at least $(2k+2)^{2^{k+1}(2^k+k+1)}$ thin  clusters, by~\eqref{cl:bound-thin-boxes}, the claim follows.
\medskip

Consider clusters $C,C'$ as in~\eqref{cl:cluster-topo-id}.
We now embed the deleted cluster $C^*$ in the box representation of $H=G-C^*$. 
For this, choose $\epsilon > 0$ small enough so that
\begin{equation}\label{choiceofeps}
\begin{minipage}[c]{0.8\textwidth}\em
 for all $v\in C$ and $w\in V(H-C)$
and all dimensions $i$ it holds that 
$|s_i-t_i|>\epsilon$, when $s$ is a corner  of the box of 
$v$ and $t$ is a corner  of
the box of $w$.
\end{minipage}\ignorespacesafterend 
\end{equation} 
(If such an $\epsilon$ does not exist, we may again perturb the box representation slightly
so as to guarantee~\eqref{perturb} while keeping~\eqref{thinify}.)

Define $q\in\mathbb R^d$ by setting
\[
q_i=\begin{cases}
1&\text{if $p_i(C)<p_i(C')$}\\
-1&\text{if $p_i(C)>p_i(C')$}\\
0&\text{if $p_i(C)=p_i(C')$}.
\end{cases}
\]
Let $v\mapsto v^*$ be the bijection between $C$ and $C^*$ with $N_G(v)\cap X=N_G(v^*)\cap X$.
We define a box for every $v^*\in C^*$ by taking a copy of the box of $v$ and shifting its
coordinates by the vector~$\epsilon\cdot q$, that is, for every dimension~$i$ we set 
\[
\ell_i(v^*)=\ell_i(v)+\epsilon q_i\text{ and }r_i(v^*)=r_i(v)+\epsilon q_i
\]
Note that, by choice of $\epsilon$, the box of $v^*$ and the box of $v$ are in the same position.

Let~$\tilde{G}$ be the graph defined by this new box representation. We claim that $\tilde G= G$,
which then finishes the proof of the lemma.

To prove this, we first note that we only added edges between vertices in $C^*$ and $H$,
while all other adjacencies remain unchanged.
Next, as
$p(C)+\epsilon q$ is a point that lies in every
box of $C^*$, it follows that $\tilde G[C^*]$ is a complete graph. 
Moreover, by choice of $\epsilon$,
we have \[N_{\tilde{G}}(v^*) \setminus (C \cup C^*) = N_{G}(v) \setminus (C \cup C^*)\]
for any $v\in C$.
In particular, $N_{\tilde{G}}(v^*) \cap C'=\emptyset$.
It remains to show that also $N_{\tilde{G}}(v^*) \cap C=\emptyset$.

For this, let $w^* \in C^*$ and $v \in C$ be arbitrary, 
where we allow that $v=w$.
Let us show that the boxes of $v$ and $w^*$ do not intersect.

Since $v$ and $w'$ are nonadjacent in $H$,
there is a dimension $i$ such that either
$r_i(v)<\ell_i(w')$ or $r_i(w')<\ell_i(v)$.
By symmetry, we may assume $r_i(v)<\ell_i(w')$.
Let $I$ be the grid interval such that $r_i(v)\in I$.
If $\ell_i(w')\notin I$,
then $r_i(v)<\ell_i(w^*)$, since by our construction $\ell_i(w^*)$ is in the same grid interval as $\ell_i(w')$.
This means that
the boxes of $v$ and $w^*$ do not intersect.
Thus, we may assume that $\ell_i(w^*)\in I$.
As $v$ and $w$ are in the same cluster and thus adjacent, 
it follows that $\ell_i(w)\leq r_i(v)$, which implies
that $p_i(C)\in [\ell_i(w),r_i(v)]\subseteq I$. 
Now, \eqref{thinify}
implies that $r_i(v)=p_i(C)=\ell_i(w)$. 

Since $p_i(C)=r_i(v)<\ell_i(w')$, it follows that $p_i(C)<p_i(C')$. 
Thus, $r_i(v)=\ell_i(w)<\ell_i(w)+\epsilon=\ell_i(w^*)$. 
Consequently,
the boxes of $v$ and $w^*$ do not intersect.
This completes the proof.
\end{proof}

We can now prove the main lemma of this section.

\begin{proof}[Proof of Lemma~\ref{dtckernel}]
\sloppy 
Let $I=(G=(V,E), b)$ be an instance of~\boxicity with cluster vertex deletion number~$k$.
We first compute a set $X$ of size $|X|\leq 3k$, so that $G-X$ is a cluster graph.
To this end we use the fact that a graph is a cluster graph if and only if it does not contain a~$P_3$ as an induced subgraph. 
Start with~$X = \emptyset$. If~$G$ contains three vertices $v_1,v_2, v_3$ that induce a~$P_3$ in~$G$ then add these three vertices to~$X$. Reiterate the process on~$G \setminus \{v_1,v_2, v_3\}$ until no more induced~$P_3$ is found. Clearly any optimal solution needs to delete at least one vertex in an induced~$P_3$.
Thus $|X| \leq 3k$.

Next, we iteratively remove one twin of any pair of true twins from $G$ until the graph becomes
free of true twins. By Lemma~\ref{notwinslem}, this does not change the boxicity of $G$. 
In the next step, we divide the clusters of $G-X$ into their equivalence classes which can be done in polynomial time. 
Then we delete clusters from 
every equivalence class until each equivalence class has at most $2(2k+2)^{2^{k+1}(2^k+k+1)}$ members. 
Since every cluster has size at most $2^{|X|} \le 2^{3k}$, by Lemma~\ref{smallclusters}, the resulting 
graph $H$ has size at most  $k^{2^{O(k)}}$. Moreover, Lemma~\ref{clusterdellem} shows that $\boxOp(H)=\boxOp(G)$. 
This completes the proof.
\end{proof}

\section{Proof of Theorem~\ref{pwalgothm}}\label{sec:pwalgothm}

Bounded pathwidth suggest a dynamic programming approach, and this is precisely what we do. 
There is a hitch, though. The standard approach would be to solve the \boxicity problem
on one bag after another of the path decomposition, so that the local solutions can be 
combined to a global one. \boxicity, however, does not permit this: 
as we are constructing the box representation of the graph, we may have to completely rearrange the previous boxes to add a new one. 

Thus, the key issue is to force  the problem to become ``localized''. 
To this end, we introduce a special interval graph~$I^*$ that reflects the path structure of the graph: two vertices are adjacent if and only if they appear in the same bag of the path decomposition. Doing so, we can safely compute local box representions of the subgraphs induced by the bags without paying attention to how these representations overlap. Indeed, the interval graph~$I^*$ gets rid of any unwanted adjacency. 

Having sketched the idea, we now give the formal description of the algorithm.
We say that two interval graphs $I=([\ell_v,r_v])_{v\in V}$ and $I'=([\ell'_{v'},r'_{v'}])_{v'\in V'}$
are \emph{consistent} if the order of the interval endpoints of the common vertices
is the same, that is, if
for all $u,v\in V\cap V'$  
\begin{align*}
\ell_u\leq \ell_v \Leftrightarrow \ell'_u\leq \ell'_v;\quad
r_u\leq r_v\Leftrightarrow r'_u\leq r'_v;\\
 \ell_u\leq r_v\Leftrightarrow \ell'_u\leq r'_v;\text{ and }r_u\leq \ell_v\Leftrightarrow r'_u\leq \ell'_v.
\end{align*}
In particular, if $V'\subseteq V$ then $I'$ is an induced subgraph of $I$.
When we consider tuples $(I_1,\ldots,I_d)$ and $(I'_1,\ldots,I'_d)$ of interval graphs
we say that they are \emph{consistent} if $I_i$ is consistent with $I'_i$ for 
$i=1,\ldots,d$.

\begin{lemma}\label{consistencylem}
If $I=([\ell_v,r_v])_{v\in V}$ and $I'=([\ell'_{v'},r'_{v'}])_{v'\in V'}$ are two consistent interval graphs, 
then there is an interval graph $J$ on $V\cup V'$ that is consistent
with both $I$ and $I'$.
In particular,  any edge in $E(J)\sm (E(I)\cup E(I'))$ 
has one endvertex in $V\sm V'$ and the other in $V'\sm V$.
\end{lemma}
\begin{proof}
Pick $v^*\in V'\sm V$ and apply induction to the pair of consistent interval graphs
$I$ and $I'-v^*$ in order to obtain an interval graph $J'$ on $(V\cup V')\sm\{v^*\}$
that is consistent with $I$ and with $I'$. 
Let $J'=([\tilde{\ell}_v,\tilde{r}_v])_{v\in  (V\cup V')\sm\{v^*\}}$.
Among the interval endpoints of $I'-v^*$, that is in $\bigcup_{v'\in V'\sm\{v^*\}}\{\ell'_{v'},r'_{v'}\}$,
we pick the two consecutive endpoints $p,q$ for which $p\leq \ell'_{v^*}\leq q$. Notice that the endpoint~$p$ (resp.~$q$) might not exist; in such case we simply set~$p = \ell'_{v^*}$ (resp.~$q = \ell'_{v^*}$). Let~$p'$ (resp.~$q'$) be the corresponding endpoint of~$p$ (resp.~$q$) in~$J'$.
Then we put $\tilde{\ell}_{v^*}=\tfrac{1}{2}(p'+q')$ and 
define $\tilde{r}_{v^*}$ in the analogous way. 
Observe that we have~$\tilde{\ell}_{v^*} \leq \tilde{r}_{v^*}$ by the choice of~$p,q$.
Adding the resulting interval $[\tilde{\ell}_{v^*},\tilde{r}_{v^*}]$ to~$J'$
yields an interval graph $J$ on $V\cup V'$ that is consistent with both $I$ and $I'$.

The second assertion of the lemma follows from the first. 
\end{proof}

The following observation is obvious:
\begin{lemma}\label{gridifylem}
If $I$ is an interval graph on $V$ as vertex set, 
then there is an interval graph $J$ on $V$ that is consistent with $I$
and so that each of the intervals of $J$ has its endpoints in $\{1,\ldots,2|V|\}$.
\end{lemma}

      \begin{figure}[ht]
      \centering
      \includegraphics[scale=0.7]{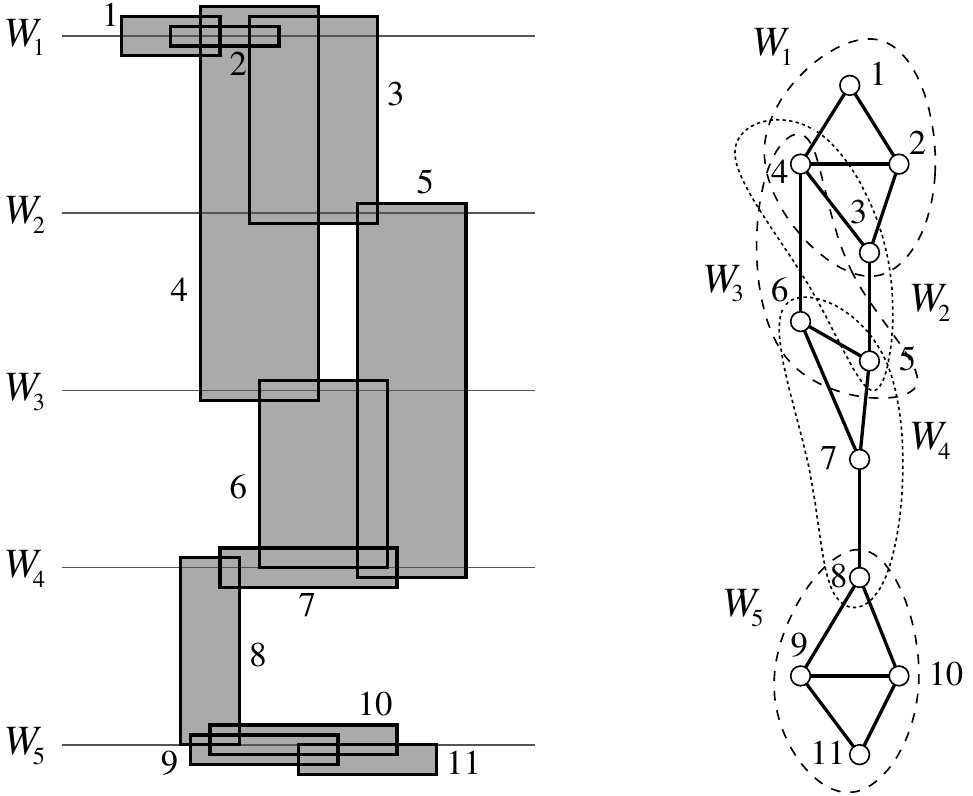}
      \caption{A box representation computed by the algorithm,
where $d=1$.}\label{fig:extradim}
      \end{figure}

We can now state the main result of this section.

\begin{theorem}\label{pwboxthm}
There is an algorithm that, for any 
graph $G$ with a given path decomposition of width~$w$, determines 
in $2^{O(w^2\log w)}\cdot |V(G)|$-time a $d\in\mathbb N$ so that $d\leq \boxOp(G)\leq d+1$.
\end{theorem}
\begin{proof}
By omitting duplicated or empty bags, we may assume the given path decomposition~$\mathcal{W}=\{W_1,\ldots,W_t\}$ of~$G$ 
to have length $t\leq |V(G)|$. 
By Theorem~\ref{boxtwthm}, 
the boxicity of $G$ is at most $w+2$. Thus, it suffices to describe an algorithm 
that checks for some fixed $d\leq w+2$, whether $G$ has a box representation of
dimension at most $d+1$. This algorithm is then executed for every $d=1,2,\ldots, w+2$.

Our algorithm proceeds as follows.
\begin{enumerate}[1.]
\item Put $\mathcal B_0=\{(\emptyset,\ldots,\emptyset)\}$, which we consider
as a tuple of $d$ empty interval graphs.
\item For $s$ from 1 to $t$ do the following.
\begin{enumerate}[a.]
\item Compute all tuples of interval graphs $(J_1,\ldots,J_d)$, where 
each $J_i$ is an interval graph on the vertex set $W_s$ so that all the interval endpoints are in $\{1,\ldots,2w+2\}$,
and where $G[W_s]=\bigcap_{i=1}^dJ_i$.
\item For each of these tuples $(J_1,\ldots,J_d)$, check whether there is a tuple $(J'_1,\ldots,J'_d)\in\mathcal B_{s-1}$ that 
is consistent with $(J_1,\ldots,J_d)$. If yes, add $(J_1,\ldots,J_d)$ to $\mathcal B_s$.
\item If $\mathcal B_s=\emptyset$ exit with $\boxOp(G)> d$.
\end{enumerate}
\item Exit with $\boxOp(G)\leq d+1$.
\end{enumerate} 

For the running time, note that the loop of line~2 is executed $t\leq |V(G)|$ times.
Each execution of a--c requires $2^{O(w^2\log w)}$-time. Indeed, each interval graph $J_i$
as in line~2a has vertex set $W_s$, and can thus be described by at most $2(w+1)$ numbers, all of which 
are in $\{1,\ldots,2w+2\}$. Thus, there are at most $(2w+2)^{2w+2}$ possible 
such interval graphs $J_i$ and therefore at most $(2w+2)^{d(2w+2)}$ tuples 
considered in line~2a. Thus, each execution of lines~2a--c can be performed within the above claimed running time.

\bigskip
To verify that the algorithm is correct, we introduce a bit of notation and state two claims.
Let us define an interval graph $I^*=([\ell^*_v,r^*_v])_{v\in V(G)}$, where
\[
\ell^*_v=\min_{v\in W_i}i\text{ and }r^*_v=\max_{v\in W_j}j.
\]
We may perturb these points slightly, so that all endpoints become distinct.
Note that $I^*$ is the interval graph induced by the path decomposition of $G$;
that is $u,v\in V(G)$ are adjacent in $I^*$ if and only if there is a bag $W_s$
such that $u,v\in W_s$.
Furthermore, we denote by $\Gsegment{s}$ the  induced subgraph $G[\bigcup_{s'=1}^sW_{s'}]$
of $G$ on the first $s$ bags of the path decomposition. 

Inductively, we prove two claims:
\begin{equation}\label{ourboxes}
\begin{minipage}[c]{0.8\textwidth}\em
for every $(J_1,\ldots,J_d)\in\mathcal B_s$ there 
is a tuple $(I_1,\ldots,I_d)$ of interval graphs 
that is consistent with $(J_1,\ldots,J_d)$ and so that 
$\Gsegment{s}=\bigcap_{i=1}^dI_i\cap I^*$.
\end{minipage}\ignorespacesafterend 
\end{equation} 
and 
\begin{equation}\label{lowerbnd}
\begin{minipage}[c]{0.8\textwidth}\em
if there  are interval graphs $I_1,\ldots,I_d$ with 
$\Gsegment{s}=\bigcap_{i=1}^dI_i$ then there is a tuple
$(J_1,\ldots,J_d)\in\mathcal B_s$ that is consistent with $(I_1,\ldots,I_d)$.
\end{minipage}\ignorespacesafterend 
\end{equation} 
Before proving the claims, we show how it follows from them that the algorithm is correct.
Consider the case when the algorithm stops in line~3.
Then $\mathcal B_t\neq\emptyset$, and with Claim~\eqref{ourboxes} for $s=t$, 
we obtain $G=\Gsegment{t}$ as the intersection of $d+1$ interval graphs, 
which proves that $\boxOp(G)\leq d+1$. 
Now suppose the algorithm 
exits in line~2c, that is, that there is an $s$ with $\mathcal B_s=\emptyset$.
From Claim~\eqref{lowerbnd} we deduce that there is no tuple $(I_1,\ldots,I_d)$
of interval graphs with $\Gsegment{s}=\bigcap_{i=1}^dI_i$, which then 
precludes the existence of interval graphs $I'_1,\ldots,I'_d$ with $G=\bigcap_{i=1}^dI'_i$,
as their restrictions $I'_i[\bigcup_{s'=1}^sW_s]$ would intersect to $\Gsegment{s}$. 
Thus,  $\boxOp(G)>d$, which finishes the proof of correctness. 
\medskip

We first prove Claim~\eqref{ourboxes}.
For this, consider $(J_1,\ldots,J_d)\in\mathcal B_s$. 
If $s=1$ then $\bigcap_{i=1}^d J_d=G[W_1]=\Gsegment{1}$, and~\eqref{ourboxes} is satisfied
by setting $I_i=J_i$. 

Hence let $s>1$. 
By definition of $\mathcal B_{s}$
there is a tuple $(J'_1,\ldots,J'_d)\in\mathcal B_{s-1}$ that is consistent with 
$(J_1,\ldots,J_d)$. Induction yields a tuple of interval graphs 
$(I'_1,\ldots I'_d)$ that is consistent with $(J'_1,\ldots,J'_d)$, and for which 
$\Gsegment{s-1}=\bigcap_{i=1}^d I'_i\cap I^*$.
Applying Lemma~\ref{consistencylem} to each consistent pair $I'_i$ and $J_i$
yields an interval graph $I_i$ consistent with both $I'_i$ and $J_i$. Note that $I_i$ is a
supergraph of both~$I'_i$ and~$J_i$. Since $G[W_s]=\bigcap_{i=1}^d J_i$
this means that $\Gsegment{s}$ is a subgraph of $\bigcap_{i=1}^d I_i\cap I^*$.

Consider $i\in\{1,\ldots,d\}$ and an edge $e$ of $I_i$ that 
is neither an edge of $I'_i$, nor of $J_i$. From Lemma~\ref{consistencylem} it follows
that $e$ has an endvertex $u$ in $V(I'_i-J_i)$ and another endvertex $v$
in $V(J_i-I'_i)$. 
Hence $u\in W_{s'}\sm W_s$ for some $s'<s$ and $v\in W_s\sm W_{s-1}$.
In particular, the corresponding intervals, $[\ell^*_u,r^*_u]$ and $[\ell^*_v,r^*_v]$, of $I^*$
do not intersect. Thus
\(
I_i\cap I^* = (I'_i\cup J_i)\cap I^*.
\)
Now,  $\Gsegment{s-1}=\bigcap_{i=1}^d I'_i\cap I^*$
together with
$G[W_s]=\bigcap_{i=1}^d J_i$ implies 
$\Gsegment{s}=\bigcap_{i=1}^d I_i\cap I^*$, as desired.

\medskip
Finally, we show~\eqref{lowerbnd}. If $s=1$ then $\bigcap_{i=1}^d I_i=G[W_1]$,
and thus, by Lemma~\ref{gridifylem}, the algorithm computes in line~2a a tuple $(J_1,\ldots,J_d)$ of
interval graphs that is consistent with $(I_1,\ldots,I_d)$. 
Since every such tuple is consistent with $(\emptyset,\ldots,\emptyset)\in\mathcal B_0$,  
it is added to $\mathcal B_1$ in line~2b. 

Hence consider now $s>1$.
Letting $I'_i$ be the restriction of $I_i$ on $\bigcup_{s'=1}^{s-1}W_{s'}$, we see that
induction yields a tuple $(J'_1,\ldots,J'_d)\in\mathcal B_{s-1}$
that is consistent with $(I'_1,\ldots,I'_d)$.
Next, we apply Lemma~\ref{gridifylem} to $I_i[W_s]$ in order to obtain 
an interval graph $J_i$ that is consistent with $I_i[W_s]$,
and whose intervals have all their endpoints
in $\{1,\ldots,w+2\}$. Consequently, the tuple $(J_1,\ldots,J_d)$ is among the tuples computed
in step~2a of the algorithm. 
Moreover, $(J_1,\ldots,J_d)$ is consistent with 
$(I'_1,\ldots,I'_d)$, and thus, also consistent with $(J'_1,\ldots,J'_d)\in\mathcal B_{s-1}$.
Therefore, $(J_1,\ldots,J_d)$ is added to $\mathcal B_s$ in step~2b of the algorithm.
\end{proof}

We mention that, while the algorithm as given only computes the number $d$, 
we can also recover a concrete box representation of dimension $d+1$. For this, it 
suffices to store for each tuple in $\mathcal B_s$ to which tuple in $\mathcal B_{s-1}$
it is consistent (if there are more, we simply choose one).

\sloppy Together with the algorithm of Bodlaender~\cite{bodlaender96-b} that computes a path-decomposition of a graph~$G$ of width~$\pwOp(G)$ in~$f(\pwOp(G)) \cdot |V(G)|$ time, we obtain Theorem~\ref{pwalgothm}.
We note that the running time could conceivably be improved by using a faster approximation algorithm
with, say, a constant approximation factor.

\section{Proof of Theorem~\ref{bandthm}}\label{sec:bandthm}

It is an open problem whether boxicity is polynomial-time solvable on graphs of bounded treewidth.
While we cannot solve the problem, we can offer
an indication why we suspect boxicity  to be hard. 

The first approach to prove tractability is usually dynamic programming. Evidently, 
this is because Courcelle~\cite{courcelle90} proved that a vast number of problems, 
namely those expressible in monadic second order logic, can be solved in polynomial time
by a generic
dynamic programming algorithm, if the treewidth is bounded. 
However, nobody appears to know how to formulate ``$\boxOp(G)\leq d$?'' in monadic second order logic, and 
it is doubtful that this is possible at all. More generally, dynamic programming 
seems  to fail. Why is that so? We think this is because the tree-like structure
of the input graph does not translate to a tree-like structure in the interval representation:
given an input graph $G$ of bounded treewidth, it may very well be 
the case that at least one interval graph in any optimal 
interval representation of~$G$ has unbounded treewidth.  

To illustrate this, 
consider a $K_{2,n}$, where the smaller bipartition class is comprised of 
two vertices $x$ and $y$, and the larger consists of $v_1,\ldots,v_n$.
Clearly, $K_{2,n}$ has pathwidth~$2$ and boxicity~$2$
as well: in fact, $K_{2,n}+xy$ and $K_{2,n}+\{v_iv_j:i,j\}$ are two interval graphs
whose intersection is $K_{2,n}$. Now, let $I_1,I_2$ be any two interval graphs 
with $K_{2,n}=I_1\cap I_2$. The vertices $x$ and $y$ are not adjacent in at least
one of $I_1$ and $I_2$, say in $I_1$. Suppose that $I_1$ contains a pair of non-adjacent $v_i,v_j$:
then $xv_iyv_jx$ is an induced $4$-cycle, which is impossible in an interval graph. 
Thus, $\{v_i\}_{i=1}^n$ form a clique of size $n$ in $I_1$, and $I_1$ has therefore pathwidth
at least $n-1$.

\medskip
What about stronger width-parameters?
We have found a similar, albeit  more complicated, example for bounded bandwidth,
a parameter even more restrictive than  pathwidth.
Theorem~\ref{bandthm} is a direct consequence of the following lemma.

\begin{lemma}\label{spirallem}
For every $n$ there is a graph $G^n$ of bandwidth at most~$16$ and boxicity~$2$, 
so that in any interval representation $G=I_1\cap I_2$ one of $I_1$ and $I_2$
has treewidth $\geq |V(G^n)|/32$.
\end{lemma}

In light of the lemma, we would like to strengthen the conjecture of Adiga~et al.~\cite{adiga10-b}:
We believe that \boxicity remains NP-complete even for graphs of bounded bandwidth.

\begin{proof}[Proof of Lemma~\ref{spirallem}]
As a basic building block for $G^n$
we use copies of the graph $B$, which consists of a path 
$w_1\ldots w_6$ and two vertices $u,v$ adjacent to each of $w_1,\ldots,w_6$ 
but not to each other; see Figure~\ref{fig:spiralgadget}. 
Clearly, $B$ has boxicity~$2$, and moreover, if $B$
is represented as the intersection of two interval graphs $I_1,I_2$ then
\begin{equation}\label{uvalternative}
\text{\em
for some $k \in \{1,2\}$,  $uv\in E(I_k)$ and $w_1,\ldots, w_6$ is a clique in $I_{3-k}$. 
}
\end{equation}
This follows directly from the fact that $uw_ivw_{j}u$ is an induced $4$-cycle in $B$
for each $1\leq i< j-1 < 6$.

     \begin{figure}[ht]
      \centering
\input{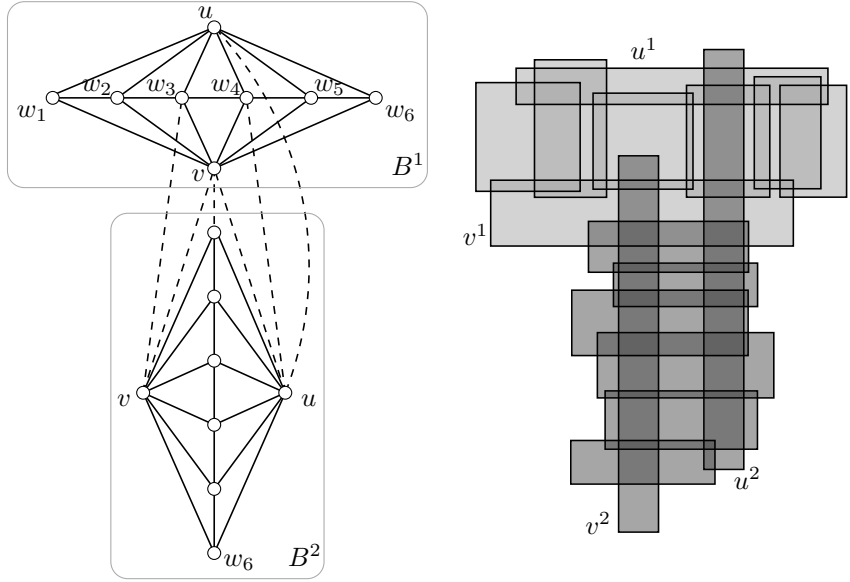}
      \caption{Glueing the gadgets (left) and a geometric realisation (right)}\label{fig:spiralgadget}
      \end{figure}

We define graphs $G^1,G^2,\ldots,G^n$ iteratively by taking as $G^1$ a copy $B^1$ of $B$ 
with vertex set $u^1,v^1,w_1^1,\ldots,w^1_6$. Then, given $G^i$ we obtain $G^{i+1}$
by adding another copy $B^{i+1}$ of $B$ on vertex set $u^{i+1},v^{i+1},w_1^{i+1},\ldots,w^{i+1}_6$,
where we make 
\begin{itemize}\itemsep1pt\parskip0pt 
\item $u^{i+1}$ adjacent to $u^i,v^i$ and $w_4^i$;
\item $v^{i+1}$ adjacent to $v^i$ and $w_3^i$; and
\item $w_1^{i+1}$ adjacent to $v^i$.
\end{itemize}
An indication that $G^n$ has indeed boxicity~$2$ as claimed is given in Figure~\ref{fig:spiralgadget}.
That $G^n$ has bandwidth~$\leq 16$ can also easily be checked. 
We fix two interval graphs $I_1,I_2$, so that 
$G^n=I_1\cap I_2$.

First, we prove that the edge $uv$ flips between consecutive copies of $B$, that is
\begin{equation}\label{uvflip}
\text{\em
for some $k\in\{1,2\}$, $u^iv^i\in E(I_k)$ implies $u^{i+1}v^{i+1}\in E(I_{3-k})$.
}
\end{equation}
To keep notation simple, we consider the case when $i=1$, and we assume that $u^1v^1\in E(I_1)$,
so that our task is to show that $u^2v^2\in E(I_2)$.

Observe that for each $j=1,2,3,5,6$ the vertices $u^1w^1_jv^1u^2u^1$ form a $4$-cycle in $G^n$
and thus in $I_2$. Since $I_2$ is chordal but $u^1$ and $v^1$ are not adjacent in $I_2$, it
follows that $u^2w^1_j\in E(I_2)$. As this edge is not present in $G^n$, 
we have consequently that  
\begin{equation}\label{preY}
\text{\em\ $w^1_4$ is the only neighbour in $I_1$ of $u^2$ among $w^1_1,\ldots,w^1_6$.}
\end{equation}

The graph (a) in Fig.~\ref{fig:forbidden} we call a $Y$, the graph (b) we call an {\it umbrella}.
As both $I_1$ and $I_2$ are interval graphs, neither of them contains a $Y$ or an {\it umbrella} as an induced subgraph.

We next show that 
\begin{equation}\label{goodedge}
v^1w^2_2\in E(I_1).
\end{equation}
From~\eqref{preY}, we deduce that in $I_1$ the induced path $w^1_2\ldots w^1_6$ together with $u^2$
and $w^2_2$ forms an induced $Y$, 
unless $w^2_2$ is adjacent to at least one of the vertices 
$w^1_2,\ldots,w^1_6$. 
Thus, we can choose $j$ be such $w^1_j$ is a neighbour of $w^2_2$ in $I_1$ and closest to $w^1_4$. If $j\neq 4$ then 
$w^1_j\ldots w^1_4u^2w^2_2$ forms an induced cycle of length $\geq 4$, which contradicts the chordality of $I_1$. 
Thus, $w^2_2w^1_4$ is an edge of $I_1$. 

Now $w^1_2\ldots w^1_6$ together
with $v^1$ and $w^2_2$ form an induced umbrella in $I_1$, 
unless $w^2_2$ has further neighbours among $w^1_2,\ldots, w^1_6$, or $w^2_2v^1\in E(I_1)$.
In the latter case, we have proved~\eqref{goodedge}, so assume that in $I_1$,
$w^2_2$ has a second neighbour in $\{w^1_2,\ldots, w^1_6\}$. As $I_1$ is chordal
one of $w^1_3$ and $w^1_5$ has to be adjacent to $w^2_2$. 
By symmetry, let us say that $w^1_3w^2_2\in E(I_1)$. 
Then $w^1_3w^2_2u^2v^1w^1_3$ is a $4$-cycle that is induced unless $w^2_2$ and $v^1$ are
adjacent in $I_1$ (here we use~\eqref{preY} again). This proves~\eqref{goodedge}.

Turning back to $G^n$, we observe that $v^1u^2w^2_2v^2v^1$ form an induced $4$-cycle. Thus,
each of $I_1$ and $I_2$ must contain exactly one of the diagonals $v^1w^2_2$ and $u^2v^2$.
Since we have already proved that $v^1w^2_2\in E(I_1)$ it follows that $u^2v^2\in E(I_2)$.
This finishes the proof of~\eqref{uvflip}.

\medskip
Next, we see that consecutive copies of $B$ rotate by $90^{\circ}$. More formally, 
we associate with a gadget $B^i$ a matrix $D^i\in\mathbb R^{2\times 2}$,
where the first row encodes the orientation of the induced path 
$w^i_1\ldots w^i_6$, while the second row corresponds to $u,v$. If $u^iv^i\in E(I_k)$
for $k=1$ or $k=2$, we set
\[
D^i_{1k}=\begin{cases}
1&\text{if }\ell_{k}(w^i_1)<\ell_{k}(w^i_6)\\
-1&\text{otherwise}
\end{cases}
\text{ and } 
D^i_{1,3-k}=0
\] and
\[
D^i_{2k}=0\text{ and }
D^i_{2,3-k}=\begin{cases}
1&\text{if }r_{3-k}(v^i)<\ell_{3-k}(u^i)\\
-1&\text{otherwise}
\end{cases}
\]
By symmetry, we may assume $B^1$ to be embedded in such a way that 
$u^1v^1\in E(I_1)$, 
$r_2(v^1)<\ell_2(u^1)$
and $\ell_1(w^1_1)<\ell_1(w^1_6)$.
Then $D^1$ is simply the identity matrix. Below we see that $D^i$
is the $(i-1)$th power of the rotation matrix
$R:=
\bigl(\begin{smallmatrix}
0&-1\\ 1&0
\end{smallmatrix} \bigr)$.

We only prove $D^2=R$. For larger $i$ this follows with analogous arguments. 
As we assumed $u^1v^1\in E(I_1)$, it follows from by~\eqref{uvflip}
that $u^2v^2\in E(I_2)$, which implies $D^2_{11}=0$ and $D^2_{22}=0$.
Let us next show that $D^2_{21}=1$, that is
\begin{equation}\label{firstturn}
r_1(v^2)<\ell_1(u^2).
\end{equation}
Indeed, as $u^2v^2\in E(I_2)$,  we either have 
$r_1(v^2)<\ell_1(u^2)$ or $r_1(u^2)<\ell_1(v^2)$. 
By~\eqref{uvalternative}, $w_1^1\ldots w_6^1$ is an induced path in $I_1$, 
and from the assumption that 
$\ell_1(w^1_1)<\ell_1(w^1_6)$ we deduce  $\ell_1(w^1_3)< \ell_1(w^1_4)$.
From~\eqref{preY} it follows that $\ell_1(u^2)\in (r_1(w^1_3),r_1(w^1_4)]$. 
Since $v^2$ is adjacent to $w^1_3$ we now see that 
$r_1(u^2)<\ell_1(v^2)$ is impossible as this would imply 
$r_1(w^1_3)<r_1(u^2)<\ell_1(v^2)$. This shows~\eqref{firstturn}.

To show the rotation, it remains to prove that 
\begin{equation}\label{secondturn}
 \ell_2(w^2_1)>\ell_2(w^2_6),
\end{equation}
which is to say that $D^2_{12}=-1$.

Suppose that $ r_2(v^1)\leq r_2(w^2_1)$. Then since $r_2(v^1)< \ell_2(u^1)$
it follows that each of $w^1_1,\ldots, w^1_6$ is adjacent to $w^2_1$ in $I_2$, 
from which we deduce that $w^2_1$ is not a neighbour of any of $w^1_1,\ldots, w^1_6$
in $I_1$. However, by~\eqref{preY},
then $w^1_2\ldots w^1_6$ induces with $u^2$ and $w^2_1$
a $Y$-subgraph in $I_1$, which is impossible. Thus $r_2(w^2_1)< r_2(v^1)$.

Now, if $\ell_2(w^2_1) <\ell_2(w^2_6)$ then one of $w^2_2,\ldots,w^2_6$
must be adjacent to $v^1$ in $I_2$. Since $I_2$ is a chordal graph
and $w^2_1\ldots w^2_6$ an induced path in $I_2$, this is only possible
if $v^1w^2_2$ is an edge of $I_2$. But, by~\eqref{goodedge}, 
we also have $v^1w^2_2\in E(I_1)$, which contradicts $v^1w^2_2\notin E(G^n)$.
Therefore, we have proved~\eqref{secondturn}.

\medskip
     \begin{figure}[ht]
      \centering
\input{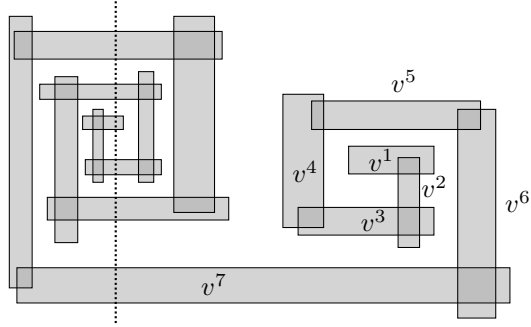}
      \caption{Spiral of $v^i$-boxes, and vertical line intersecting $\geq n/4$ boxes}\label{fig:thespiral}
      \end{figure}
Summing up, the boxes corresponding to the induced
path $v^1v^2\ldots v^n$ are arranged in a spiral pattern. This spiral has two 
options: either it may become ever larger or, after a number of steps, it may 
become smaller and smaller. In both cases, we find a vertical line 
that meets a quarter of the $v^i$-boxes, which translates to a clique of size 
$n/4$ in $I_1$.
In particular, $\twOp(I_1) \ge n/4$.
As $|V(G)| = 8n$, the proof is complete.
\end{proof}

\section{Discussion}\label{sec:discussion}

In this paper, we treated \boxicity from the perspective of parameterized complexity.
We presented 
\begin{itemize}
	\item a parameterized algorithm for \boxicity with respect to the parameter cluster vertex deletion;
	\item \sloppy an additive 1-approximation algorithm for \boxicity that runs in~$2^{O(w^2\log w)}\cdot n$ time where~$w$ is the width of a given path decomposition of the input graph;
	\item and a family of graphs of bounded bandwidth that need, in any optimal box representation, dimensions of unbounded treewidth.
\end{itemize}

In some respect, the method of our first algorithm is a generalization of 
the true twin reduction. The key insight is that if there are many
vertex sets (the clusters) that are identical in the graph then 
many of these sets will have essentially the same geometric realization. 
Deleting one of these many ``geometric twins'' is unlikely to change boxicity.

We believe this approach can exploited further. Indeed, we are convinced that
with similar methods as developed in this article, we can also formulate 
a parameterized algorithm for \boxicity when the parameter is \emph{distance to stars} -- the smallest number of vertices whose removal results 
in a disjoint union of stars.
Like cluster vertex deletion, distance to stars provides a non-trivial parameterization for \boxicity between vertex cover (solved) and feedback vertex set (open). 
Moreover, given a graph~$G$, computing a minimum set~$X \subseteq V(G)$ such that~$G[V - X]$ is a disjoint union of stars can be done in~$f(|X|) \cdot |V(G)|^{O(1)}$ time~\cite{NRT05}.

\smallskip
Our second algorithm yields an additive 1-approximation for \boxicity on graphs of bounded pathwidth.
Two questions that immediately arise are: can we get rid of the additive 1, such that the algorithm computes $\boxOp(G)$ exactly? Can the algorithm be lifted to run on graphs of bounded treewidth? 

We stated earlier our impression of the first question -- we conjecture \boxicity to be NP-complete on graphs of bounded bandwidth,
thus including graphs of bounded pathwidth.
Our reasoning is that, in any optimal representation
as the intersection of interval graphs, 
some of the interval graphs may have unbounded treewidth, even if the input graph has bounded bandwidth.
This seems to annul the main advantage of bounded pathwidth, 
that the number of box representation of the union of previous bags can be compressed
to a size bounded by a function of the pathwidth.

\medskip
We turn to the second question: why is it difficult to extend the algorithm  to graphs of bounded treewidth?
We rely heavily on the fact that the one extra dimension is sufficient to reflect 
the path decomposition of the whole graph. If we mimick this approach for bounded treewidth
we have to  describe the tree decomposition of the graph with as few extra dimensions as possible. 
How many extra dimensions would we need? As many as the boxicity of the chordal supergraph
obtained by turning each bag of the decomposition into a clique.
If we started with a path decomposition, the boxicity will be one. For a general tree decomposition, 
however, it could well be that the boxicity of this chordal graph is about the 
treewidth of the input graph~\cite{chandran07}. 
This suggests that there might be input graphs $G$ for which $\boxOp(G)$ is much lower than the number of dimensions required to describe their tree decomposition, 
which makes it impossible to approximate using only the techniques of Section~\ref{sec:pwalgothm}.

\bibliographystyle{abbrv}
\bibliography{main}

\vfill
\small
\vskip2mm plus 1fill
\noindent
Version \today
\bigbreak

\noindent
Henning Bruhn {\tt <henning.bruhn@uni-ulm.de>}\\
Morgan Chopin {\tt <morgan.chopin@uni-ulm.de>}\\
Felix Joos {\tt <felix.joos@uni-ulm.de>}\\
Institut f\"ur Optimierung und Operations Research\\
Universit\"at Ulm, Germany\\[3pt]
Oliver Schaudt
{\tt <schaudto@uni-koeln.de>}\\
Institut f\"ur Informatik\\
Universit\"at zu K\"oln, Germany

\end{document}